\documentclass{amsart}

\usepackage{lineno,hyperref}
\usepackage{url}
\usepackage{amsmath}
\usepackage{amssymb}
\usepackage{amscd}   
\usepackage{amssymb,amsmath}
\usepackage{graphicx}
\usepackage{hyperref}
\usepackage[english]{babel}
\usepackage[T1]{fontenc}       		%uncomment these two lines
\usepackage{mathrsfs}   
\usepackage{comment}
\usepackage[textsize=tiny]{todonotes}
\usepackage{xcolor}
\usepackage{pgf,tikz,pgfplots}

\usepackage{mathrsfs}
\usetikzlibrary{arrows}
\newtheorem{Definition}{Definition}[section]
\newtheorem{Theorem}[Definition]{Theorem}
\newtheorem{Proposition}[Definition]{Proposition}

\newtheorem{Remark}[Definition]{Remark}
\newtheorem{Lemma}[Definition]{Lemma}

\providecommand{\ch}[1]{\text{\raise 2pt \hbox{$\chi$}\kern-0.2pt}_{#1}}

\def\XXint#1#2#3{{\setbox0=\hbox{$#1{#2#3}{\int}$}
      \vcenter{\hbox{$#2#3$}}\kern-.5\wd0}}

\usepackage[mathcal]{eucal}					%for the set of BV subsets of Rn

\def\L{{{\mathcal{L}}}}

\renewcommand{\geq}{\geqslant}
\renewcommand{\leq}{\leqslant}
\renewcommand{\epsilon}{\varepsilon}

\usepackage{color}%
\begin{document}

\author{Emma D'Aniello and Martina Maiuriello}
\date{\today}

\title{On some generic small Cantor spaces}
\begin{abstract} 
Let $X = [0,1]^{n}$, $n \geq1$. We show that the typical (in the sense of Baire category) compact subset of $X$ is not only 
a zero dimensional Cantor space but it satisfies the property of being strongly microscopic, which is stronger than dimension zero. 
\end{abstract}
\keywords{Microscopic set, Cantor space, Baire category.}
\subjclass{28A05, 28A75, 54E52}

\maketitle

%\thanks{{\it Acknowledgement.} }

\section{Introduction}
Microscopic subsets of the real line were introduced in \cite{AP}. Since then, they have been widely 
studied (see, for instance, \cite{ADV}, \cite{BFN}, \cite{HO}, \cite{HO1},  \cite{HKW}, \cite{KW1}, \cite{KPW}, \cite{KPOW},  \cite{KW0}, \cite{KW},  \cite{FW}, \cite{PA}, \cite{PAWB}). The collection of microscopic subsets of the real line is a $\sigma$-ideal and it is a proper subset of the family of the Hausdorff zero dimensional sets (\cite{ADV}, \cite{HKW}).  In \cite{BFN}, the authors consider the family ${\mathcal CS}$ of symmetric Cantor subsets of $[0, 1]$, and among other results, they obtain properties concerning the subfamily of microscopic sets. In \cite{HO}, \cite{HO1} and \cite{KPW}, the authors generalise the notion of a microscopic set in ${\Bbb R}$. \\
In \cite{PA} and \cite{PAWB} some Fubini type properties involving microscopic fibers are studied. In \cite{KW1} and in \cite{KW}, the notion of a microscopic set in the plane is investigated.  The authors use the word ``miscroscopic" when the coverings consist of rectangles and the expression ``strong microscopic" when only coverings made of cubes are considered. \\
\indent This note is motivated by the above mentioned results. In section 2, we investigate how ``frequent" strongly microscopic, and therefore microscopic, compact sets are.    
We show that most compact subsets of ${[0,1]}^{n}$, $n \geq 1$,  are strongly microscopic. 
Our approach is the following. We furnish the class
of nonempty closed subsets of ${[0,1]}^{n}$ with the Hausdorff metric. Since it is a complete metric space (\cite{BBT}, \cite{EDG} \cite{FA}), we can 
make effective use of the Baire Category Theorem (\cite{BBT}, \cite{OX}). We find information concerning ``typical" members of in 
Theorem \ref{THMMICRO}. Indeed, we find that the typical compact subset $E$ of 
${[0,1]}^{n}$ is a strongly microscopic Cantor space, where we recall the term typical is to indicate that the collection of sets having the property under 
consideration has first category complement in the complete metric space. \\
Moreover, in section 3, we provide, in any dimension, an example of a non-microscopic, Hausdorff zero dimensional Cantor space.  

\section{Microscopic sets are not exceptional in ${[0,1]}^{n}$}

\begin{Definition}
Let $n \geq 1$. We call {\it interval} any subset $I$ of ${\Bbb R}^n$ of type $I= J_{1} \times \dots \times J_{n}$, with $J_{i}$ intervals of the real line, for each $1\leq i\leq n$. We call $I$ a {\it cube} if $J_{1}= \dots = J_{n}$.
\end{Definition}
Clearly, an interval $I$ is open if $J_{i}$ is an open interval of the real line for every $1 \leq i \leq n$. 

\begin{Definition}
A set $E \subseteq {\Bbb R}^n$ is {\it microscopic} if for each $\epsilon>0$ there exists a sequence of intervals $\left\lbrace I_k\right\rbrace _{k \in {\Bbb N}}$ such that $$E \subseteq \cup_{k \in {\Bbb N}}I_k 
\textit{ and } \lambda(I_k)\leq \epsilon^k, \textit{ for } k \in {\Bbb N},$$ where $\lambda$ is the Lebesgue measure on ${\Bbb R}^n.$
\end{Definition}

\begin{Definition}
A set $E \subseteq {\Bbb R}^n$, $n \geq 2$, is {\it strongly microscopic} if for each $\epsilon>0$ there exists a sequence of cubes $\left\lbrace I_k\right\rbrace _{k \in {\Bbb N}}$ such that $$E \subseteq \cup_{k \in {\Bbb N}}I_k 
\textit{ and } \lambda(I_k)\leq \epsilon^k, \textit{ for } k \in {\Bbb N},$$ where $\lambda$ is the Lebesgue measure on ${\Bbb R}^n.$
\end{Definition}

\begin{Remark}
Let  $E \subseteq {\Bbb R}^n$, $n \geq 2$. Then 
$$E \text{ strongly microscopic } \Rightarrow E \text{ microscopic.}$$
The above implication cannot be reverted \cite{KW}.  
\end{Remark}

Some simple properties about microscopic and strongly microscopic sets hold, and they are collected in the following proposition.
The proof in \cite{ADV} for the case $n=1$ works for any $n \in {\Bbb N}$, with slight, obvious modifications.

\begin{Proposition} \label{PROPT}
The following hold in ${\Bbb R}^n$:
\begin{enumerate}
\item{Every countable set is strongly microscopic.}
\item{Every microscopic set is a null set (meaning it has Lebesgue $n$-dimensional measure equal to $0$).}
\item{Every subset of a (strongly, resp.) microscopic set is (strongly, resp.) microscopic.}
\item{Every countable union of (strongly, resp.) microscopic sets is (strongly, resp.) microscopic.}
\item{Every strongly microscopic set $E$ has $\alpha$-dimensional Hausdorff measure equal to zero for all $\alpha>0$, and thus it has Hausdorff dimension zero.}
\end{enumerate}
\end{Proposition}

\begin{Remark}
Clearly, the notion of a strongly microscopic set differs from the notion of a microscopic set when $n\geq 2$.  When $n\geq 2$, 
Property $(5)$ of Proposition \ref{PROPT}  does not hold if we replace ``strongly microscopic" with ``microscopic". For example, 
as also observed in \cite{KW}, the set $A= [0,1] \times \{0\}$ is a microscopic subset of ${\Bbb R}^{2}$, but its Hausdorff dimension is one. 
\end{Remark}

Fix any $n \geq 1$. By ${\mathcal M}$ and by ${\mathcal M_{S}}$ we denote the collection of all microscopic subsets of ${\Bbb R}^{n}$ and the family of all strongly microscopic subsets of ${\Bbb R}^{n}$, respectively.\\

In order to show that strongly microscopic sets, and therefore microscopic sets, are not exceptional among the compact subsets of ${[0,1]}^{n}$, indeed they are very frequent, we need to recall some classical definitions and facts, and prove some preliminary results.    \\

In the sequel, by $d(\underline{x}, \underline{y})$, with $\underline{x}$ and $\underline{y}$ points of ${[0,1]}^{n}$, we always mean the Euclidean distance in ${\Bbb R}^{n}$. \\
As usual, we define the {\it diameter} of a non-empty set $A$, and we denote it by $diam(A)$, as the greatest distance apart of pairs of points in $A$, that is, $diam(A)=\sup \{d(\underline{x}, \underline{y}): \underline{x}, \underline{y} \in A\}$. Given two sets $A$ and $B$, we define their {\it Euclidean distance}, and we denote it by $dist(A,B)$ as $dist(A,B)=inf \{ d(\underline{x}, \underline{y}): \underline{x} \in A, \underline{y} \in B \}$. Hence, given a point
$\underline{z}$, the {\it distance} between the point $z$ and the set $A$ is
$dist(\underline{z}, A) =  dist( \{\underline{z} \},A)$.\\

Let ${\mathcal K}$ be the collection of non-empty, compact subsets of ${[0,1]}^{n}$, $n \geq 1$. We furnish ${\mathcal K}$ with the Hausdorff metric given by
$${\mathcal H}(A,E) = \text{inf} \{ \delta > 0: A \subset B_{\delta}(E), E \subset B_{\delta}(A)\},$$
where  the {\it $\delta$-neighbourhood} or {\it $\delta$-parallel body}, $B_{\delta}(E)$, of
a set $E$ is the set of points within distance $\delta$ of $E$. The Hausdorff metric space $({\mathcal K}, {\mathcal H})$ is also 
compact (\cite{BBT}, \cite{EDG} \cite{FA}), and therefore complete. \\
Since $({\mathcal K}, {\mathcal H})$ is complete, we can use the Baire category theorem.  Recall that a set is of the first category in the complete 
space $(X, \rho)$ whenever it can be written as a countable union of nowhere dense sets; otherwise, the set is of the second category (\cite{BBT}, \cite{OX}). 
A set is residual if it is the complement  of a first category set, and an element of a residual set 
is called {\it typical} (or generic). \\

In the sequel, by $V \left( [0,1]^n\right)$, we denote the set of all vertices of $[0,1]^n$, that is 
$$V \left( [0,1]^n\right)  =\{ (x_{1}, \dots, x_{n}) \in [0,1]^n : x_{j} \in \{0, 1\}, \forall  j \in \{1, \dots, n\}\}.$$

Given an interval $I =  J_{1} \times \dots \times J_{n}$ of ${\Bbb R}^n$, with $J_{i}$ intervals of the real line, for each $1\leq i\leq n$, we let 
$$l_{I} = \min\{ \lambda(J_{i}) : 1\leq i\leq n\}$$
where $\lambda(J_{i})$ is the $1$-dimensional Lebesgue measure of the real interval $J_{i}$, that is its length. 

\begin{Definition}
Let $I_1,...,I_t$ be open intervals (relative to $[0,1]^n$). Let $B(I_1,...,I_t)$ be the collection of all $K \in {\mathcal K}$ such that:
\begin{enumerate}
\item $K \subseteq \cup_{i=1}^t I_i$;
\item $K \cap I_i\neq\emptyset$ for each $i \in \lbrace1,...,t\rbrace$.
\end{enumerate}
\end{Definition}

Next we generalise Lemma 2.5 of \cite{BC} to any finite dimension.

\begin{Lemma} Let $I_1,...,I_t$ be open intervals (relative to $[0,1]^n$). 
Then  $B(I_1,...,I_t)$ is open in  $({\mathcal K}, {\mathcal H}) $.
\end{Lemma}
\begin{proof}
Let $K \in B(I_1,...,I_t)$. We need to find $\epsilon>0$ such that $B_{\epsilon}(K)\subseteq B(I_1,...,I_t)$ (we recall that  $B_{\epsilon}(K)= \lbrace T \in {\mathcal K}  : {\mathcal H}(T,K)<\epsilon \rbrace$). 
For each ${i}\in{\left\lbrace 1,...,t\right\rbrace}$, we choose $\underline{x_i} \in {K \cap {I_i}}$ and we define: \[\ \epsilon_i=\left\{ \begin{array}{ll}
dist\left( {\underline{x}_i}, {[0,1]^n}\setminus{I_i} \right)  & \mbox{if } \underline{x}_i \notin V\left( [0,1]^n\right),  \\ l_{I_i} & \mbox{if } 
\underline{x}_i \in V\left( [0,1]^n\right),\\
\end{array}
\right.\] and \[\ {\epsilon_0}=\left\{ \begin{array}{ll}
1  & \mbox{if } {\cup_{i=1}^{t}{I_i}}=[0,1]^n,  \\ dist\left( K, {[0,1]^n}\setminus{\cup_{i=1}^{t} {I_i}} \right) & \mbox{if } {\cup_{i=1}^{t}{I_i}}\neq[0,1]^n. \\
\end{array}
\right.\] Let $\epsilon:=min\lbrace\epsilon_i, 0\leq i\leq t\rbrace.$ We now show that $B_{\epsilon}(K)\subseteq B(I_1,...,I_t)$. To this aim, let us consider $T \in B_{\epsilon}(K)$. We first show that $T \cap I_i \neq\emptyset $ 
for each $i \in \lbrace 1,...,t\rbrace$. \\
Let $i \in \lbrace 1,...,t\rbrace$. As ${\mathcal{H}}(K,T)<\epsilon\leq\epsilon_i $, it follows that $T \cap I_i \neq\emptyset $. 
In fact:
\begin{enumerate}
\item[\emph{Case 1:}]  $\underline{x}_i \notin V \left( [0,1]^n\right)$. Then, $\epsilon_i=dist\left( {\underline{x}_i}, {[0,1]^n}\setminus{I_i} \right)$. If by contradiction $T \cap I_i =\emptyset$, then 
$T \subseteq ({[0,1]^n}\setminus{I_i})$ so that $dist(\underline{x}_i,T)>\epsilon_i$ and then $\underline{x}_i \notin B_{\epsilon_i}(T).$ This is in contrast with the fact that $\underline{x}_i \in K \subseteq B_{\epsilon_i}(T)$.
\item[\emph{Case 2:}] $\underline{x}_i \in V \left( [0,1]^n\right)$. Then, $\epsilon_i=
l_{I_{i}}$. As $\underline{x}_i \in K \subseteq B_{\epsilon_i}(T)$, it follows that there exists $\underline{z} \in T$ 
such that  $d(\underline{x}_i,\underline{z})<{\epsilon_i}= l_{I_i}$ and hence $z \in T \cap I_i $.
\end{enumerate}
Now, we prove that if ${\mathcal{H}}(K,T)<\epsilon_0 $ then $T \subseteq \cup_{i=1}^{t} {I_i}$. In fact:
\begin{enumerate}
\item[\emph{Case i:}] $\epsilon_0=1$. Then, clearly $T \subseteq \cup_{i=1}^{t} {I_i}$ since  ${\cup_{i=1}^{t} {I_i}}=[0,1]^n$.
\item[\emph{Case ii:}] $\epsilon_0=dist\left( K, {[0,1]^n}\setminus{\cup_{i=1}^{t}{I_i}}\right)$,  with $\cup_{i=1}^{t} {I_i} \not= [0,1]^n$. If, by contradiction, $T \nsubseteq \cup_{i=1}^{t} {I_i}$, then 
by the definition of $\epsilon_0$ it cannot be $T \subseteq B_{\epsilon_0}(K)$. This is in contrast with the fact that ${\mathcal{H}}(K,T)<\epsilon_0 $.
\end{enumerate}
Hence, the proof is complete.
\end{proof}

\begin{Proposition} \label{PROPA} 
The typical compact subset of $[0,1]^n$ is a strongly microscopic set.
\end{Proposition}
\begin{proof}
Let $${\mathcal M}_{{\mathcal SK}}={\left\lbrace  E\subset {[0,1]^n}  : \text{E is non-empty, compact and strongly microscopic} \right\rbrace} $$ that is $ {\mathcal M}_{{\mathcal SK}}={\mathcal K} \cap {\mathcal M_{S}} $. \\
Let us start by showing that $ {\mathcal M}_{{\mathcal SK}}$ is dense in ${\mathcal K}$. Let ${\mathcal F}=\left\lbrace E \in {\mathcal K}: E \text{ is finite}\right\rbrace $. It is clear that $\mathcal F$ is dense in ${\mathcal K}$  
and each element of $\mathcal F$ is strongly microscopic. Thus ${\mathcal M}_{{\mathcal SK}}$ is dense in  ${\mathcal K}.$ \\
Now we prove that ${\mathcal M}_{{\mathcal SK}}$ is a $G_\delta$ subset of ${\mathcal K}.$
For each $s \in {\Bbb N}$ let 
\begin{eqnarray*}
& & {\mathcal K}^{[s]} = \\
& & \left\lbrace E \in {\mathcal K}: \exists\lbrace I_j \rbrace_{j \in {\Bbb N}} \text{ sequence of open cubes with } E \subseteq \cup_{j \in {\Bbb N}} I_j, \lambda(I_j)\leq \left( \dfrac{1}{s}\right)^j\right\rbrace. \\
\end{eqnarray*}
Clearly, ${\mathcal M}_{{\mathcal SK}}= \cap_{s=1}^ \infty {\mathcal K}^{[s]}.$  Next, we show that, for every $s \in {\Bbb N}$, the set  ${\mathcal K}^{[s]}$ is open. Fix $s \in {\Bbb N}$.  Let $E \in {\mathcal K}^{[s]}$, and let  $\{I_{j}\}_{j \in {\Bbb N}}$ be a covering of $E$,  where each $I_{j}$ is an open cube with $\lambda(I_{j}) \leq {(\frac{1}{s})}^{j}$.
Then, as $E$ is compact, there exist $I_{j_1},...,I_{j_m}$ finitely many open cubes such that:
\begin{enumerate}
\item {$ E \subseteq \cup_{t=1}^m I_{j_t}$,}
\item {$\lambda(I_{j_t}) \leq \left( \dfrac{1}{s}\right) ^{j_t}$, for $t=1,...,m.$}
\end{enumerate}
Thus, if $ F \in B(I_{j_1},...,I_{j_m})$ it is $F \subseteq \cup_{t=1}^m I_{j_t}$ and hence $F \in  {\mathcal K}^{[s]}.$ So that ${\mathcal K}^{[s]}$ is open. Therefore  ${\mathcal M}_{{\mathcal SK}}$ is a dense 
$G_\delta$ set, and hence the thesis. 
\end{proof}

We recall that a topological space is a {\it Cantor space} if it is non-empty, perfect, compact, totally disconnected, and metrisable.\\
Hence, a topological space is a Cantor space if it is homeomorphic to the Cantor ternary set. 
As in \cite{DS}, let ${\Bbb Q}$ denote the rational numbers and $${\mathcal IR}=\left\lbrace (a_1,...,a_n) \in [0,1]^n : a_j \in [0,1] \setminus {\Bbb Q}, \text{ for all } 1 \leq j \leq n\right\rbrace ,$$ 
that is, ${\mathcal IR}$ denotes the collection of all points in $[0,1]^n$ with all the coordinates irrational. \\
Let, as in \cite{DS},  
$${\mathcal K}_1 =\left\lbrace F \in {\mathcal  K} : F \text{ is a Cantor space and } F\subseteq {\mathcal  IR}\right\rbrace .$$ We recall from \cite{DS} the following proposition:

\begin{Proposition} \label{PROPA1} 
The collection ${\mathcal  K}_1$ is a dense set of type $G_\delta$ in ${\mathcal  K}.$
\end{Proposition}

From this result and from Proposition \ref{PROPA} we obtain the following theorem. 

\begin{Theorem} \label{THMMICRO}
The typical compact subset $K$ of $[0,1]^n$ is a strongly microscopic Cantor space. 
\end{Theorem}
\begin{proof}
By Proposition \ref{PROPA}, the collection of the non-empty, compact and strongly microscopic subsets of $[0,1]^n$, that is ${\mathcal  M}_{{\mathcal  SK}},$ is a dense set of type $G_\delta$. 
By Proposition \ref{PROPA1}, the collection ${\mathcal  K}_1$ also is a dense $G_\delta$ set. Since the intersection of two dense $G_\delta$ sets is still a dense $G_\delta$ set, the thesis follows.
\end{proof}

\section{Examples}
Cantor spaces in ${[0,1]}^{n}$ can be of various type, for instance very  irregular, symmetric, etc. The Cantor ternary set is symmetric but it is not microscopic. On the other hand, on the real line, there exist symmetric Cantor spaces that are microscopic. More is true: in  \cite{BFN}, it is shown that in $[0,1]$ microscopic symmetric Cantor spaces form a residual family.  \\
\indent In ${[0, 1]}^{n}$, we find Cantor spaces of any positive dimension, self-similar and non, (\cite{FA}, \cite{BBT}, \cite{ED}, \cite{BFN}) and therefore non-strongly microscopic. In \cite{ADV}, 
it was shown that, in $[0,1]$,  Hausdorff dimension zero is only a necessary but not a sufficient condition in order for a set, and in particular for a Cantor space, to be microscopic. In the next example, for any $n \geq 1$, we construct a 
Cantor space in ${[0,1]}^{n}$ having Hausdorff dimension zero but not being microscopic. \\

\begin{Definition}
In the following, for a general $n \geq 1$, by a {\it cube} we mean a non-degenerated closed cube, that is any interval 
$I= J_{1} \times \dots \times J_{n}$ of ${\Bbb R}^n$ such that there exists a non-degenerate closed interval of the real line, $[a, b]$, with
$J_{i} = [a,b]$ for each $1\leq i\leq n$. Clearly, when $n=1$ it is more natural to talk about closed intervals, when $n=2$ about closed squares, and when $n=3$ about closed cubes. 
\end{Definition}

{\bf Example 3.2} Fix $n \geq 1$. In the sequel, by $\lambda$ we denote the Lebesgue $n$-dimensional measure. The Cantor space we construct from the unit cube ${[0,1]}^{n}$ is a ``Cantor dust" (\cite{EDG0}, \cite{FA}) . At each stage of the construction we select smaller cubes in the way described below.\\
Fix a constant $c\geq 2^n +1,$  let $V_0=1,$ and $V_k=\dfrac{1}{c^{k^2}}$ for $k \in {\Bbb N}.$ \\
%Let $ {\mathcal  V}=\lbrace\underline{x}=(x_1,...,x_n),\, x_i \in \lbrace 0,1 \rbrace, \textit{ for each }i \in {\lbrace 1,...,n\rbrace}\rbrace.$\\ 
At the first step we select $2^n$ disjoint cubes in $[0,1]^n$ of measure $V_1=\dfrac{1}{c},$ and having one vertex in $V \left( [0,1]^n\right)$.  We list these cubes as $Q_{i_1}$, with $i_1 \in {\lbrace 1,...,2^n \rbrace}$. \\ At the second step, in each of the previous cubes  $Q_{i_1},$ we select $2^n$ disjoint cubes of measure $V_2=\dfrac{1}{c^{4}},$ having one vertex in common with $Q_{i_1}$, and we list them as $Q_{i_1,i_2}$, with $(i_1,i_2) \in {\lbrace 1,...,2^n \rbrace^2}.$ \\ At the k-th step, in each cube $Q_{i_1,...,i_{k-1}}$ with $(i_1,...,i_{k-1}) \in  {\lbrace 1,...,2^n \rbrace^{k-1}},$ we select $2^n$ disjoint cubes of measure $V_k=\dfrac{1}{c^{k^2}}$ and having one vertex in common with $Q_{i_1,...,i_{k-1}}$. We list these cubes  as $Q_{i_1,...,i_k},$ with $(i_1,...,i_k) \in {\lbrace 1,...,2^n \rbrace^k}.$ 

%\newpage 

%costruzione albero

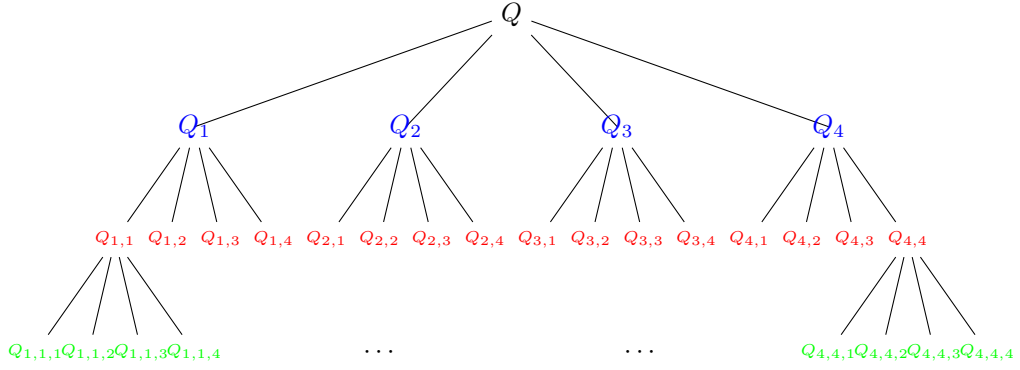
\begin{figure}
\begin{tikzpicture}[sibling distance=8em]
  \node {$Q$}
    child { [sibling distance=2em] node (a) {\color{blue}{$Q_1$}} 
     child { node {\color{red}{\tiny{$Q_{1,1}$}}}
             child { node (a1) {\color{green}{\tiny{$Q_{1,1,1}$}}}}
             child { node (a2) {\color{green}{\tiny{$Q_{1,1,2}$}}}}
             child { node (a3) {\color{green}{\tiny{$Q_{1,1,3}$}}}}
             child { node (a4) {\color{green}{\tiny{$Q_{1,1,4}$}}}}}
     child { node {\color{red}{\tiny{$Q_{1,2}$}}}}
     child { node {\color{red}{\tiny{$Q_{1,3}$}}}}
     child { node {\color{red}{\tiny{$Q_{1,4}$}}}}}
    child { [sibling distance=2em] node {\color{blue}{$Q_2$}} 
     child { node {\color{red}{\tiny{$Q_{2,1}$}}}}
     child { node {\color{red}{\tiny{$Q_{2,2}$}}}}
     child { node {\color{red}{\tiny{$Q_{2,3}$}}}}
     child { node {\color{red}{\tiny{$Q_{2,4}$}}}}}
    child { [sibling distance=2em] node {\color{blue}{$Q_3$}} 
     child { node {\color{red}{\tiny{$Q_{3,1}$}}}}
     child { node {\color{red}{\tiny{$Q_{3,2}$}}}}
     child { node {\color{red}{\tiny{$Q_{3,3}$}}}}
     child { node {\color{red}{\tiny{$Q_{3,4}$}}}}}
    child { [sibling distance=2em] node (b) {\color{blue}{$Q_4$}}
     child { node {\color{red}{\tiny{$Q_{4,1}$}}}}
     child { node {\color{red}{\tiny{$Q_{4,2}$}}}}
     child { node {\color{red}{\tiny{$Q_{4,3}$}}}}
     child { node {\color{red}{\tiny{$Q_{4,4}$}}}
               child { node (b1) {\color{green}{\tiny{$Q_{4,4,1}$}}}}
               child { node (b2) {\color{green}{\tiny{$Q_{4,4,2}$}}}}
               child { node (b3) {\color{green}{\tiny{$Q_{4,4,3}$}}}}
               child { node (b4) {\color{green}{\tiny{$Q_{4,4,4}$}}} }}};
    \path (a4) -- (b1) node [midway] {$\cdots \hspace{3cm} \cdots $};
   \end{tikzpicture}
\caption{Tree of the construction of Example 3.2 in $Q=[0,1]^2$}
\label{Fig1}%
\end{figure}

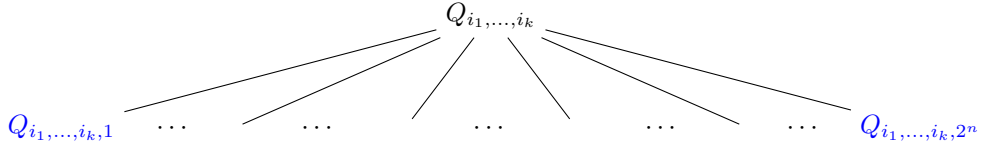
\begin{figure}
\begin{tikzpicture}[sibling distance=6.5em]   
     \node {$Q_{i_1,...,i_k}$}
    child { node (c) {\color{blue}{$Q_{i_1,...,i_{k},1}$}}} 
    child { node (e) {\color{blue}}}
 % child { dashed (f) {\color{blue}}}
    child { node (f) {\color{blue}}} 
    child { node (g) {\color{blue}}}
    child { node (h) {\color{blue}}}
    child { node (i) {\color{blue}{$Q_{i_1,...,i_{k},2^n}$}}};
  \path (c) -- (e) node [midway] {$ \cdots $};
 \path (e) -- (f) node [midway] {$ \cdots $};
    \path (f) -- (g) node [midway] {$ \cdots $};
 \path (g) -- (h) node [midway] {$ \cdots $};
 \path (h) -- (i) node [midway] {$ \cdots $};
    \end{tikzpicture}
\caption{Step $k$ to step $k+1$ in the construction of Example 3.2 in $Q=[0,1]^2$}
\label{Fig1}%
    \end{figure}

%costruzione cubo

\begin{figure}
\begin{center}
\color{cyan}
\setlength{\unitlength}{10cm}
\begin{picture}(1,1)(0,0)

%Assi
%\put(0,0){(0,0)}
\put(0,1){(0,1)}
\put(1,0){(1,0)}
\put(1,1){(1,1)}
\put(0,1){\line(1,0){1}}
\put(0,0){\line(0,1){1}}
\put(1,0){\line(0,1){1}}
\put(0,0){\line(1,0){1}}

%Costruzione quadrati step 1
%\put(0.45,0){(0.45,0)}
\put(0.45, 0){\line(0,1){0.45}}

%\put(0.55, 0){(0.55, 0)}
\put(0.55,0){\line(0,1){0.45}}

%\put(0,0.45){(0,0.45)}
\put(0, 0.45){\line(1,0){0.45}}

%\put(0, 0.55){(0, 0.55)}
\put(0, 0.55){\line(1,0){0.45}}

%\put(0.45, 1){(0.45, 1)}
\put(0.45, 0.55){\line(0,1){0.45}}

%\put(0.55, 1){(0.55, 1)}
\put(0.55, 0.55){\line(0,1){0.45}}

%\put(1, 0.55){(1, 0.55)}
\put(0.55, 0.55){\line(1,0){0.45}}

%\put(1, 0.45){(1, 0.45)}
\put(0.55, 0.45){\line(1,0){0.45}}

%costruzione quadrati step 2

\put(0.04,0){\line(0,1){0.04}}
\put(0.41,0){\line(0,1){0.04}}
\put(0.59,0){\line(0,1){0.04}}
\put(0.96,0){\line(0,1){0.04}}
\put(0,0.04){\line(1,0){0.04}}
\put(0.41,0.04){\line(1,0){0.04}}
\put(0.55,0.04){\line(1,0){0.04}}
\put(0.96,0.04){\line(1,0){0.04}}

\put(0.04,0.41){\line(0,1){0.04}}
\put(0.41,0.41){\line(0,1){0.04}}
\put(0.59,0.41){\line(0,1){0.04}}
\put(0.96,0.41){\line(0,1){0.04}}
\put(0,0.41){\line(1,0){0.04}}
\put(0.41,0.41){\line(1,0){0.04}}
\put(0.55,0.41){\line(1,0){0.04}}
\put(0.96,0.41){\line(1,0){0.04}}

\put(0.04,0.55){\line(0,1){0.04}}
\put(0.41,0.55){\line(0,1){0.04}}
\put(0.59,0.55){\line(0,1){0.04}}
\put(0.96,0.55){\line(0,1){0.04}}
\put(0,0.59){\line(1,0){0.04}}
\put(0.41,0.59){\line(1,0){0.04}}
\put(0.55,0.59){\line(1,0){0.04}}
\put(0.96,0.59){\line(1,0){0.04}}

\put(0.04,0.96){\line(0,1){0.04}}
\put(0.41,0.96){\line(0,1){0.04}}
\put(0.59,0.96){\line(0,1){0.04}}
\put(0.96,0.96){\line(0,1){0.04}}
\put(0,0.96){\line(1,0){0.04}}
\put(0.41,0.96){\line(1,0){0.04}}
\put(0.55,0.96){\line(1,0){0.04}}
\put(0.96,0.96){\line(1,0){0.04}}

%Denominazioni
\put(0.22, 0.77){\color{blue}{$Q_1$}}
\put(0.77, 0.77){\color{blue}{$Q_2$}}
\put(0.22, 0.22){\color{blue}{$Q_3$}}
\put(0.77, 0.22){\color{blue}{$Q_4$}}

\put(0.003, 0.97){\color{red}{\tiny{$Q_{1,1}$}}}
\put(0.413, 0.97){\color{red}{\tiny{$Q_{1,2}$}}}
\put(0.003, 0.56){\color{red}{\tiny{$Q_{1,4}$}}}
\put(0.413, 0.56){\color{red}{\tiny{$Q_{1,3}$}}}

\put(0.553, 0.97){\color{red}{\tiny{$Q_{2,1}$}}}
\put(0.963, 0.97){\color{red}{\tiny{$Q_{2,2}$}}}
\put(0.553, 0.56){\color{red}{\tiny{$Q_{2,3}$}}}
\put(0.963, 0.56){\color{red}{\tiny{$Q_{2,4}$}}}

\put(0.003, 0.423){\color{red}{\tiny{$Q_{3,1}$}}}
\put(0.413, 0.423){\color{red}{\tiny{$Q_{3,2}$}}}
\put(0.003, 0.01){\color{red}{\tiny{$Q_{3,3}$}}}
\put(0.413, 0.01){\color{red}{\tiny{$Q_{3,4}$}}}

\put(0.553, 0.423){\color{red}{\tiny{$Q_{4,1}$}}}
\put(0.963, 0.423){\color{red}{\tiny{$Q_{4,2}$}}}
\put(0.553, 0.01){\color{red}{\tiny{$Q_{4,3}$}}}
\put(0.963, 0.01){\color{red}{\tiny{$Q_{4,4}$}}}

%\put(0.75, 0.75){\Large{\color{cyan}{$Q_{1}$}}}
%\put(0.51, 0.51){\large{\color{magenta}{$E_{1}$}}}
%\put(0.375, 0.375){{\color{cyan}{$Q_{2}$}}}
%\put(0.26, 0.26){\small{\color{magenta}{$E_{2}$}}}
%\put(0.1875, 0.1875){\Small{\color{cyan}{$Q_{3}$}}}
%\put(0.126, 0.126){\tiny{\color{magenta}{$E_{3}$}}}

%inserire puntini
%\put(0.02, 0.02){$\cdot$}
%\put(0.03, 0.03){$\cdot$}
%\put(0.43, 0.02){$\cdot$}
%\put(0.44, 0.03){$\cdot$}
%\put(0.57, 0.02){$\cdot$}
%\put(0.58, 0.03){$\cdot$}
%\put(0.97, 0.02){$\cdot$}
%\put(0.98, 0.03){$\cdot$}

%\put(0.02, 0.42){$\cdot$}
%\put(0.03, 0.43){$\cdot$}
%\put(0.43, 0.42){$\cdot$}
%\put(0.44, 0.43){$\cdot$}
%\put(0.57, 0.42){$\cdot$}
%\put(0.58, 0.43){$\cdot$}
%\put(0.97, 0.42){$\cdot$}
%\put(0.98, 0.43){$\cdot$}

%\put(0.02, 0.56){$\cdot$}
%\put(0.03, 0.57){$\cdot$}
%\put(0.43, 0.56){$\cdot$}
%\put(0.44, 0.57){$\cdot$}
%\put(0.57, 0.56){$\cdot$}
%\put(0.58, 0.57){$\cdot$}
%\put(0.97, 0.56){$\cdot$}
%\put(0.98, 0.57){$\cdot$}

%\put(0.02, 0.97){$\cdot$}
%\put(0.03, 0.98){$\cdot$}
%\put(0.43, 0.97){$\cdot$}
%\put(0.44, 0.98){$\cdot$}
%\put(0.57, 0.97){$\cdot$}
%\put(0.58, 0.98){$\cdot$}
%\put(0.97, 0.97){$\cdot$}
%\put(0.98, 0.98){$\cdot$}
\end{picture}
\end{center}
\caption{Construction of the cubes for $n=2$}
\label{Fig3}%
\end{figure}
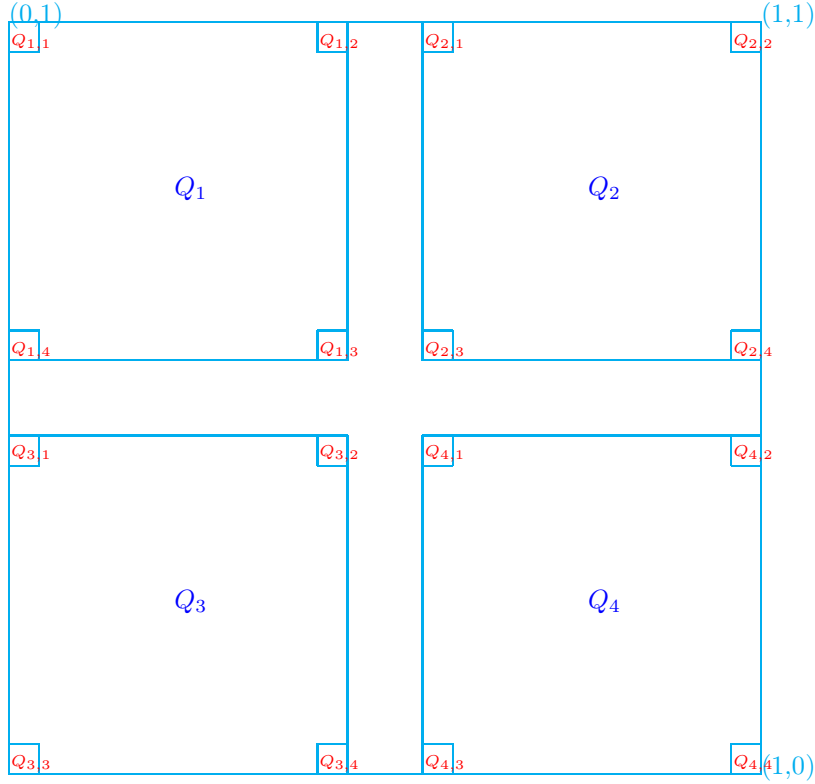

Now, we define a sequence ${\lbrace \delta_k \rbrace}_{k \in {\Bbb N}}$ as follows: 
$$\delta_1=V_0-2^nV_1=1-\dfrac{2^n}{c};$$ $$\vdots$$ $$ \delta_k=V_{k-1}-2^nV_k=\dfrac{1}{c^{(k-1)^2}}-\dfrac{2^n}{c^{k^2}};  $$ $$\vdots$$ Notice that $\delta_k$ is the measure of what is left, at the $k$-th step,  in $Q_{i_1,...,i_{k-1}}$ when we remove from it $2^n$ disjoint cubes each of measure $V_k$. This works by the choice of $c\geq 2^n +1$, from which it follows that $\delta_k>0$ for each $k \in {\Bbb N}$. \\Therefore, for each $k \in {\Bbb N} $, at the k-th step we have constructed a collection of cubes  $$ {\mathcal  Q}_k= \lbrace Q_{i_1,...,i_k} \text{, with }(i_1,...,i_k) \in {\lbrace 1,...,2^n \rbrace^k}\rbrace $$ such that: 
\begin{enumerate}
\item{each cube $Q_{i_1}$  contains a vertex of $V \left( [0,1]^n\right)$ and, for $k>1$, each cube $Q_{i_1,...,i_k}$ has only one vertex in common with $Q_{i_1,...,i_{k-1}}$;}
\item{$\lambda(Q_{i_1,...,i_k})=V_k=\dfrac{1}{c^{k^2}}$;}
\item{for each $(i_1,...,i_{k-1},j_k),(i_1,...,i_{k-1},j'_k)$, with $j_k,j'_k \in  \lbrace1,...,2^n\rbrace$, 
$$dist(Q_{i_1,...,i_{k-1},j_k},Q_{i_1,...,i_{k-1},j'_k}) \geq \dfrac{1}{\sqrt[n]{c^{(k-1)^2}}}-\dfrac{2}{\sqrt[n]{c^{k^2}}}$$
 and, if we set $d_k=\dfrac{1}{\sqrt[n]{c^{(k-1)^2}}}-\dfrac{2}{\sqrt[n]{c^{k^2}}}$, then $d_k^n$ is the Lebesgue measure of the largest cube contained in $Q_{i_1,...,i_{k-1}} \setminus \cup_{i_k \in {\lbrace1,...,2^n\rbrace}}Q_{i_1,...,i_{k-1},i_k}$, having inner points in common with at most one cube 
$Q_{i_1,...,i_k}$, with $i_{k} \in {\lbrace1,...,2^n\rbrace}.$ }
\item{for each $(i_1,...,i_{k}),(j_1,...,j_k) \in {\lbrace1,...,2^n\rbrace}^k $, 
$$dist(Q_{i_1,...,i_k},Q_{j_1,...,j_k}) \geq  D_k,$$ 
where 
$$D_1=d_1,$$ 
$$\text{ for $k\geq2$}, D_k=min\{d_k,D_{k-1}\},$$ 
and where $D_k^n$ turns out to be  the Lebesgue measure of the largest cube contained in ${[0,1]}^n \setminus \cup_{(i_1,...,i_k) \in {\lbrace1,...,2^n\rbrace}^k} Q_{i_1,...,i_k}$,  having inner points in common with at most one cube  $Q_{i_1,...,i_k}$, with $(i_{1}, \dots, i_{k}) \in {\lbrace1,...,2^n\rbrace}^{k}.$}
\end{enumerate}

Let $$C=\cap_{k \in {\Bbb N}} \cup_{(i_1,...,i_k) \in {\lbrace1,...,2^n\rbrace}^k}Q_{i_1,...,i_k}.$$ Clearly, by construction, $C$ is a Cantor space.\\
Let $\alpha>0$. We now prove that ${\mathcal {H}}^{\alpha}(C)=0$. \\
Consider the  map $f(x)=2^{nx} \dfrac{1}{c^{\frac{\alpha}{n} x^2}}=e^{nxlog2}e^{- \frac{\alpha}{n} x^2logc}.$ Clearly, $\lim _{x\rightarrow +\infty} f(x)=0.$  Observe that 
$${\mathcal {H}}^{\alpha}(C) \leq 2^{nk} \sqrt{n^{\alpha}} V_k^{\frac{\alpha}{n}},$$
where $\sqrt{n} V_k^{\frac{1}{n}}$ is the diameter of the cubes at the $k-th$ step. So that 
$$0  \leq   \lim_{k\rightarrow+\infty} {\mathcal {H}}^{\alpha}(C)  \leq   \lim_{k\rightarrow+\infty} 2^{nk} \sqrt{n^{\alpha}} V_k^{\frac{\alpha}{n}} 
  =   \sqrt{n^{\alpha}} \lim_{k\rightarrow+\infty} f(k)=0.$$
Hence, we conclude that ${\mathcal {H}}^{\alpha}(C)=0$ for each $\alpha>0$ and thus  $C$ has Hausdorff dimension 0. 

It remains to prove that $C$ is not microscopic. To this end, observe that 
\begin{eqnarray*}
 d_k & = & \dfrac{1}{\sqrt[n]{c^{(k-1)^2}}}-\dfrac{2}{\sqrt[n]{c^{k^2}}}\\
&  = & \dfrac{\sqrt[n]{c^{2k-1}}-2}{\sqrt[n]{c^{k^2}}} \\
& \geq  & \dfrac{\sqrt[n]{(2^n +1)^{2k-1}}-2}{\sqrt[n]{c^{k^2}}} \\
& \geq & \dfrac{\sqrt[n]{(2^n +1)}-2}{\sqrt[n]{c^{k^2}}},\\
\end{eqnarray*}
 where $0<{\sqrt[n]{(2^n +1)}-2}\leq 1$.
Moreover, since for each $k \geq 1$ there exists $s \in \{1,...,k \}$ for which $D_k=d_s$, then 
$$D_k=d_s \geq \dfrac{\sqrt[n]{(2^n +1)}-2}{\sqrt[n]{c^{s^2}}}.$$
Let $\epsilon= \dfrac{\left[ \sqrt[n]{(2^n +1)}-2\right] ^{4n}}{{c^{4}}}$. Assume that $\lbrace I_h \rbrace_{h \in {\Bbb N}}$ is a sequence of rectangles such that $\lambda(I_h) \leq  \epsilon^h.$ If $h$ belongs to the set $$ H_k=\left\lbrace  h \in {\Bbb N} : \dfrac{(k+1)^2}{4} > h \geq \dfrac{k^2}{4} \right\rbrace, $$  we have 
\begin{eqnarray*}
\lambda(I_h) & \leq & \epsilon^h < \epsilon^{\frac{k^2}{4}} \\ 
&=& \left( \dfrac{\left[ \sqrt[n]{(2^n +1)}-2\right] ^{4n}}{{c^{4}}}\right)^{\frac{k^2}{4}} = \dfrac{\left[ \sqrt[n]{(2^n +1)}-2\right] ^{nk^2}}{{c^{k^2}}} \\
& \leq &  \dfrac{\left[ \sqrt[n]{(2^n +1)}-2\right] ^{n}}{{c^{k^2}}} \leq   \dfrac{\left[ \sqrt[n]{(2^n +1)}-2\right] ^{n}}{{c^{s^2}}} \leq D_k^n. 
\end{eqnarray*}

Hence, $I_h$ can intersect at most $2^{(n-1)k}$ cubes remaining at the $k-th$ step. Let $a_{k}$ be the number of cubes constructed at the $k-th$ step which intersect some $I_{h}$ with 
$h < \frac{{(k+1)}^{2}}{4}$, that is some $h \in \cup_{i=1}^{k} H_{i}$. The following recursive estimate holds

$$a_{k+1} \leq 2^{n}a_{k} + 2^{(n-1)k} \vert H_{k+1} \vert \hspace{1.2cm} (\star).$$

Now,  we show that, for each $k \geq 1$,  
$$a_{k} \leq 2^{nk} - [2^{(n-1)k}]k. \hspace{2.3cm} (\bullet).$$ 
By construction, we see that it holds for $k=1, 2, 3$. In order to prove that $(\bullet)$ also holds for $k \geq 4$, we argue by induction with basis $k_{0} =3$ and applying $(\star)$.

For $k \geq 4$, 
$$\dfrac{(k+1)^2}{4} - \dfrac{k^2}{4} = \dfrac{2k+1}{4} < k-1.$$
The above estimate  means that $H_k$  contains at most  $k-2$ elements, and hence $\vert H_{k+1} \vert \leq k-1.$ 
Therefore, for $k \geq 4$, 
$$a_{k+1} \leq 2^n a_k +  [2^{(n-1)k}] [k-1].$$
Hence, if we assume 
$$a_{k} \leq 2^{nk} -  [2^{(n-1)k}]k,$$
then we have 
\begin{eqnarray*}
a_{k+1} & \leq &  2^n a_k + [2^{(n-1)k}] [k-1]\\
& \leq & 2^{n} \{2^{nk} -  [2^{(n-1)k}] k\} +  [2^{(n-1)k}] [k-1] \\
& = & 2^{n(k+1)} - [2^{n + (n-1)k}] k +  [2^{(n-1)k}] [k-1] \\
& = & 2^{n(k+1)} -  2^{(n-1)k} [k(2^{n} - 1) + 1] \\
& \leq & 2^{n(k+1)} -  2^{(n-1)k} [2^{n-1} k + 2^{n-1}] \\
& = & 2^{n(k+1)} -  2^{(n-1)(k+1)} [k+1].
\end{eqnarray*}
Therefore, $(\bullet)$ holds for every $k \geq 1$. 

The inequality $(\bullet)$, in particular, means that there are cubes at the $k-th$ step that do not intersect any of the rectangles  $I_{h}$ with $h < \frac{{(k+1)}^{2}}{4}$.  \\
For each fixed $k$, let $F_k$ be the union of these cubes. By construction, $F_k$ is a decreasing sequence of nonempty compact sets. Consequently, the intersection $F=F_1 \cap F_2 \cap F_3 ...$ is a nonempty compact subset of $C$ and, also by construction, none of the rectangles $I_h$ intersects $F$. In particular, the union of the rectangles $I_h$ cannot cover $F,$ and then cannot cover $C$ either. Hence $C$ is not microscopic.

\vspace{0.3cm}
\noindent
\small{\textsc{Emma D'Aniello},}
\small{\textsc{Dipartimento di Matematica e Fisica},}
\small{\textsc{Universit\`a degli Studi della Campania ``Luigi Vanvitelli''},}
\small{\textsc{Viale Lincoln n. 5, 81100 Caserta,}}
\small{\textsc{Italia}}\\
\footnotesize{\texttt{emma.daniello@unicampania.it}.}

\vspace{0.3cm}
\noindent
\small{\textsc{Martina Maiuriello},}
\small{\textsc{Dipartimento di Matematica e Fisica},}
\small{\textsc{Universit\`a degli Studi della Campania ``Luigi Vanvitelli''},}
\small{\textsc{Viale Lincoln n. 5, 81100 Caserta,}}
\small{\textsc{Italia}}\\
\footnotesize{\texttt{martina.maiuriello@unicampania.it}.}

\end{document}